\documentclass{article}
\usepackage{preamble}

\newkeytheorem{definition}[name = Definition, numberwithin=section]
\newkeytheorem{problem}[name = Problem]
\newkeytheorem{theorem}[name = Theorem, style = theorem, numberwithin=section]
\newkeytheorem{proposition}[name = Proposition, style = theorem, numberwithin=section]
\newkeytheorem{corollary}[name = Corollary, style = theorem, numberwithin=section]
\newkeytheorem{lemma}[name = Lemma, style = theorem, numberwithin=section]
\newkeytheorem{proof}[name = Proof, style = proof, numbered=no]

\newcommand{\cL}{\mathcal L}
\newcommand{\fL}{\mathfrak L}

\title{Meromorphic order from Taylor coefficients, least singular values of Hankel matrices and rational approximants}
\author{Kei Beauduin}
\date{}

\begin{document}

\maketitle

\begin{abstract}
	We introduce a natural notion of generalized order for entire and meromorphic functions and express it in terms of their Taylor coefficients. More precisely, one formula uses the least singular values of Hankel matrices of the coefficients, and the other uses their determinants. These formulas generalize the classical coefficient formula for the order of entire functions and its extensions by Edrei and Sheremeta. The proof combines a theorem of Pommerenke with a bound for the Nevanlinna characteristic that results from a convergence theorem for rational approximants introduced by Prokhorov. We also characterize the largest class of growth scales for which the associated generalized order is well-behaved.
\end{abstract}

\setcounter{tocdepth}{3}\tableofcontents

\section{Introduction}

Let $(c_n)_{n\in\Z}$ be a complex sequence indexed by the integers such that $c_n = 0$ for $n < 0$, and let
\[
f(z) := \sum_{n = 0}^\infty c_n z^n
\]
be the associated power series. When $f$ is an entire function, we can assign it an \emph{order} defined by
\begin{equation}\label{e:orddef}
	\rho = \rho(f) := \inf\cbra{\varrho > 0 : f(z) = O\pa{e^{|z|^\varrho}}} \in [0, \infty],
\end{equation}
which measures the growth of $f$. Building on partial results of Poincaré \cite{poincare1883}, Hadamard proved in 1893 a now classical expression for order in terms of the Taylor coefficients of $f$ \cite{hadamard1893,levin1996}.
\begin{theorem}[Poincaré--Hadamard, 1893]\label{t:poincaré}
	If $f$ is entire,
	\begin{equation}\label{e:rhoE}
		\rho = \limsup_{n\to\infty} \frac{n \log n}{\log(1/|c_n|)}.
	\end{equation}
\end{theorem}

In 1901, Borel \cite{borel1901} extended the notion of order to meromorphic functions on the complex plane. This leads to the following question.

\begin{problem}\label{pb:order}
	Can we express the order of a meromorphic function in terms of its Taylor coefficients?
\end{problem}

At first sight, the answer might seem negative: unlike Taylor series of entire functions, Laurent series of meromorphic functions are intrinsically local. Nevertheless, their coefficients uniquely determine the function, and thus its properties on the entire plane. To formulate our answer, consider the \emph{Hankel determinants}
\[
H_{n,m} := \det (c_{n+i - j})_{0 \le i, j < m} = \begin{vmatrix} c_n & c_{n - 1} & \cdots & c_{n - m+1} \\
c_{n+1} & c_n & \cdots & c_{n - m+2} \\
\vdots & \vdots & \ddots & \vdots \\
c_{n+m - 1} & c_{n+m - 2} & \cdots & c_n \end{vmatrix}.
\]
They yield the following positive answer to \Cref{pb:order}.
\begin{theorem}\label{t:rhoH}
	If $f$ is meromorphic on the complex plane,
	\begin{equation}\label{e:rhoH}
		\rho = \limsup_{n,m\to\infty} \frac{nm\log \min(n, m)}{\log(1/|H_{n,m}|)}.
	\end{equation}
\end{theorem}

\section{Preliminary remarks}

\subsection{Kronecker's theorem and Hadamard's theorem}

Kronecker's theorem \cite{kronecker1881} characterizes rational functions by the vanishing of Hankel determinants. The following is a quantitative form \cite{aptekarev2011}.
\begin{theorem}[Kronecker, 1881]
	Let $N, m\in\N$. $H_{n,m} = 0$ for every $n \ge N$ if and only if $f = P/Q$ where $\deg P < N$, $\deg Q < m$ and $Q(0) \neq 0$.
\end{theorem}

\begin{definition}
	For $m\in\N$, let $R_m$, called the \emph{$m$th radius of meromorphy}, be the radius of the largest disk centered at the origin in which $f$ admits a meromorphic continuation with at most $m$ nonzero poles (counted with multiplicity).
\end{definition}

Define
\begin{equation}\label{e:ell}
	\ell_m := \limsup_{n\to\infty} \sqrt[n]{|H_{n,m}|}.
\end{equation}
By the usual convention on empty determinants, $H_{n,0} = 1$ and thus $\ell_0 = 1$.

The $0$th radius of meromorphy $R_0$ is simply called the \emph{radius of convergence} of $f$. It was first given explicitly by Cauchy \cite{cauchy1821}, and later rediscovered by Riemann \cite{laugwitz1994} and Hadamard \cite{hadamard1888,hadamard1892}, who popularized the formula. It is now known as the \emph{Cauchy--Hadamard theorem}.
\begin{theorem}[Cauchy, 1821]\label{t:cauchy}
	The radius of convergence is given by
	\[
	R_0 = \inv{\ell_1} = \inv{\limsup\limits_{n\to\infty} \sqrt[n]{|c_n|}}.
	\]
\end{theorem}

Hadamard \cite{hadamard1889,hadamard1892} generalized the result as follows \cite{aptekarev2011}.
\begin{theorem}[Hadamard, 1889]\label{t:hadamard}
	For $m\in\N$,
	\[
	R_m = \frac{\ell_m}{\ell_{m+1}}.
	\]
	The quotient is interpreted as $\infty$ whenever $\ell_{m+1} = 0$.
\end{theorem}
For a proof, see \cite{hadamard1892,mandelbrojt1927,dienes1957,gragg1973}. Entire functions correspond to $R_0 = \infty$ whereas meromorphic functions on $\C$ are characterized by
\[
R_\infty := \lim_{m\to\infty} R_m = \infty.
\]

\Cref{t:hadamard} shows how information beyond the disk of convergence can be reached via the Hankel determinants.

\subsection{Order of entire and meromorphic functions}

As an alternative to \cref{e:orddef}, one can define the order of an entire function $f$ by
\begin{equation}\label{e:ordM}
	\rho = \limsup_{r\to\infty}\frac{\log^+ \log^+ M(r, f)}{\log r},
\end{equation}
where
\begin{equation}\label{e:maxmod}
	M(r, f) := \sup_{|z| = r} |f(z)| = \sup_{|z| \le r} |f(z)|,
\end{equation}
is the maximum modulus function and $\log^+ (x) := \max(0, \log x)$, with $\log^+ 0 := 0$.

Borel defined the order $\rho$ of $f$ as the least possible maximum of the orders of entire functions $g$ and $h$ satisfying $f = g/h$ \cite{borel1901,borel1903}. This representation misleadingly suggests that the study of meromorphic functions reduces to the study of the entire functions $g$ and $h$, but this approach turned out to be unsuited to the study of meromorphic functions \cite[p.~10]{lehto1982}.

A decisive advance in the theory of meromorphic functions was Nevanlinna's introduction of the \emph{Nevanlinna characteristic} in 1924 \cite{nevanlinna1924}:
\begin{equation}\label{e:nevanT}
	T(r, f) := m(r, f) + N(r, f),
\end{equation}
where
\[
m(r, f) := \inv{2\pi} \int_0^{2\pi} \log^+ |f(re^{i\theta})| \d\theta, \quad N(r, f) := \sum_{0 < |a|\le r} \log\pa{\frac{r}{|a|}},
\]
are the proximity function and the Nevanlinna counting function. Here the $a$ are the poles of $f$ repeated according to multiplicity.

Borel's order of a meromorphic function $f$ can be recovered by the formula \cite{nevanlinna1924}
\begin{equation}\label{e:rho}
	\rho = \limsup_{r\to\infty} \frac{\log^+ T(r, f)}{\log r}.
\end{equation}
This definition coincides with \cref{e:ordM} for entire functions. Informally, $T(r, f)$ plays for meromorphic functions the role that $\log^+ M(r, f)$ plays for entire functions.

\subsection{Survey of partial results}

A first clue comes from the well-known identity \cite[eq.~(2.5)]{pommerenke1969}
\begin{equation}
	H_{n,m}(1/f) = (-1)^{nm} H_{m,n}(f)
\end{equation}
when $f(0) = 1$. Assuming \cref{e:rhoH}, the symmetry in $n$ and $m$ then gives $\rho(1/f) = \rho(f)$, even when $f(0) \neq 1$. This property follows from Borel's and Nevanlinna's definitions.

Let $(z_m)_{m\in\N}$ be a sequence of all the poles of $f$, repeated according to multiplicity, such that the sequence $(|z_m|)_{m\in\N}$ is nondecreasing. Then let
\[
\lambda_\infty := \inf \cbra{\lambda > 0 : \sum_{m = 0}^\infty |z_m|^{-\lambda} < \infty}
\]
be the exponent of convergence. We have \cite[\S 47]{markushevich1965}
\[
\lambda_\infty = \limsup_{m\to\infty} \frac{\log m}{\log |z_m|}.
\]
Note that in fact, $|z_m| = R_m$, the meromorphy radius. Straightforward estimates and \Cref{t:hadamard} show that
\[
\lambda_\infty = \limsup_{m\to\infty} \frac{\log m}{\log \sqrt[m]{R_0 \cdots R_{m - 1}}} = \limsup_{m\to\infty} \frac{m \log m}{\log(1/\ell_m)}
\]
Since $\lambda_\infty$ cannot exceed the order $\rho$ of $f$, the definition \eqref{e:ell} yields
\[
\rho \ge \limsup_{m\to\infty} \limsup_{n\to\infty} \frac{mn \log m}{\log(1/|H_{n,m}|)}.
\]

A different lower bound was obtained by Wilson \cite[Thm.~2]{wilson1935}:
\[
\rho \ge \limsup_{m\to\infty} \frac{m^2 \log m}{\log(1/|H_{n+m-1,m}|)},
\]
for every $n\in\N$. Wilson asks whether the inequality can be improved to an equality. For a special class of meromorphic functions defined by continued fractions, Li \cite[Cor.~1.3]{li1996} obtained equality, but the general question was answered in the negative by Edrei \cite[Thm.~IV]{edrei1940}.

Applying Wiman--Valiron theory to the generating functions of the Hankel determinants, Edrei found a formula for the order of an entire function $f$ in terms of the determinants \cite[Rmk.~1]{edrei1980}.
\begin{theorem}[Edrei, 1980]\label{t:edrei}
	If $f$ is entire and of order $\rho$, then
	\begin{equation}\label{e:edrei}
		\rho = m \limsup_{n\to\infty} \frac{n \log n}{\log(1/|H_{n,m}|)},
	\end{equation}
	for any $m\ge 1$.
\end{theorem}
\Cref{t:edrei} reduces to \Cref{t:poincaré} for $m = 1$. We record two informal remarks:
\begin{itemize}
    \item \Cref{t:edrei} suggests that in the doubly indexed sequence $(|H_{n,m}|)_{n,m\in\N}$, the growth in $n$ encodes the growth of the function $|f|$. The growth in $m$, however, encodes the distribution of the poles of $f$, in view of \Cref{t:hadamard}. These two behaviors correspond precisely to the two components $m(r, f)$ and $N(r, f)$ of the Nevanlinna characteristic \eqref{e:nevanT}.
    \item Formula \cref{e:rhoH} does not immediately reduce to \Cref{t:poincaré} or \Cref{t:edrei} for entire functions. This again suggests that the study of meromorphic functions is significantly more delicate than that of entire functions.
\end{itemize}

For functions with finitely many poles, Reczek expressed order and type via Hankel determinant ratios \cite[Prop.~3.1]{reczek1984}, and later via coefficients of Padé numerators \cite[Thm.~1]{reczek1990}.

\subsection{Proof of the first inequality}

The first inequality
\begin{equation}\label{e:ineq1}
	\rho \ge p := \limsup_{n,m\to\infty} \frac{nm\log \min(n, m)}{\log(1/|H_{n,m}|)}.
\end{equation}
quickly follows from a result of Pommerenke \cite[Thm.~4]{pommerenke1969}.\footnote{There is an oversight in the reference: the denominator should be $r^{mn}$, not $r^{m+n}$.}
\begin{theorem}[Pommerenke, 1969]\label{t:pommerenke}
	If $f$ satisfies $f(0) = 1$, then for every $0 < r < R_\infty$,
	\begin{equation}\label{e:pomme}
		|H_{n,m}| \le \frac{e^{(n+m)T(r, f)}}{r^{nm}}.
	\end{equation} 
\end{theorem}

The normalization at the origin can be omitted in \cref{e:pomme}: apply \cite[Thm.~3]{pommerenke1969} to a quotient $f = h/g$ of bounded holomorphic functions with $g(0) = 1$ and $\max(\abs{h},\abs{g}) \le e^{T(r, f)}$ on $\abs{z} = r$.

Inequality \cref{e:pomme} is sharp, but significantly weaker than the analogous inequality for $0 < r < R_0$ \cite[Lem.~1]{edrei1940}
\begin{equation}\label{e:HMbound}
	|H_{n,m}| \le m! \frac{M(r, f)^m}{r^{nm}},
\end{equation}
which reduces to Cauchy's inequality when $m = 1$.

\begin{proof}[of \cref{e:ineq1}]
	If $f$ is of order $\rho$ on $\C$, then $R_\infty = \infty$ and $T(r,f) \le C r^\varrho$ for any $\varrho > \rho$ and some $C > 0$. The value of $r$ that minimizes the function $r \mapsto e^{(n+m) C r^{\varrho}}/r^{nm}$ is
	\[
	r = \pa{\frac{1}{C\varrho} \frac{nm}{n + m}}^{\inv\varrho}.
	\]
	Therefore
	\[
	|H_{n,m}| \le \pa{e \varrho C \pa{\inv n + \inv m}}^{\frac{nm}{\varrho}} \le \pa{\frac{2e \varrho C}{\min(n, m)}}^{\frac{nm}{\varrho}}
	\]
	It follows that $p \le \varrho \to \rho$.
\end{proof}

\subsection{Overview and methodology}

Motivated by some of the results presented above, the author conjectured the existence of a formula expressing meromorphic order in terms of Taylor coefficients. The precise formula \eqref{e:rhoH} emerged from LLM-assisted computations on specific examples, and the author then established the first inequality \eqref{e:ineq1}. The relationship between Hankel determinants and \emph{Padé approximants} \cite{aptekarev2011} suggested a route to the reverse inequality: approximate $f$, establish a bound involving the determinants, and pass to the limit. Neither the author's attempts nor subsequent attempts with LLMs produced a proof using this approach.

We later revisited the problem with a newer model, and it obtained a proof of the reverse inequality using instead the rational approximants recently introduced by Prokhorov \cite{prokhorov2015}. These approximants, introduced in \Cref{s:prokhorov}, minimize a Euclidean norm. The use of the $2$-norm seems retrospectively natural in light of Pommerenke's proof of \Cref{t:pommerenke} \cite{pommerenke1969}, which relies on \emph{Parseval's identity} and \emph{Hadamard's inequality}.

The initial proof was short but crude and offered little insight. With further LLM assistance, the author developed the surrounding theory. In \Cref{s:cv}, we establish a convergence theorem for Prokhorov's approximants and use it to bound the Nevanlinna characteristic in terms of the least \emph{singular values} of Hankel matrices. In \Cref{s:order}, we introduce a generalized notion of order, motivated by examples and known results, and apply this bound to obtain coefficient formulas for entire and meromorphic functions. We prove the maximality of the associated growth class in \Cref{s:max}.

\section{Least singular value of the Hankel matrix}

\subsection{Prokhorov's rational approximants}\label{s:prokhorov}

In this section, assume that $f(z) = \sum_{k=0}^\infty c_k z^k$ converges and put $c_k = 0$ for $k < 0$. Let $s_{n,m}$ be the least singular value of $A_{n,m} := (c_{n+i-j})_{0 \le i, j < m}$. Since the product of the singular values equals the modulus of the determinant, we have the inequality
\begin{equation}\label{e:sHbound}
	s_{n, m} \le \abs{\det A_{n, m}}^{\inv m} = \abs{H_{n, m}}^{\inv m}
\end{equation}

Let $Q = (q_0, \dots, q_m)$ be a tuple of $m + 1$ complex numbers associated with a polynomial of degree at most $m$
\[
Q(z) := \sum_{k = 0}^m q_k z^k.
\]
By the Courant--Fischer theorem \cite[Thm.~7.3.8]{horn2012}, we have
\begin{equation}\label{e:smin}
	s_{n+1, m+1} = \min_{\norm{Q}_2 = 1} \norm{A_{n+1, m+1} Q}_2,
\end{equation}
where $\norm{}_2$ denotes the Euclidean norm. Let $Q_{n,m}$ be a tuple attaining the minimum and define the polynomial $P_{n,m}$ by the partial Taylor sum
\begin{equation}\label{e:pnm}
	Q_{n,m}(z) f(z) = P_{n,m}(z) + O(z^{n+1}), \quad z \to 0.
\end{equation}
Therefore, $P_{n,m}$ and $Q_{n,m}$ have degree at most $n$ and $m$, respectively. \emph{Prokhorov's approximant} is the rational function $R_{n,m} := P_{n,m}/Q_{n,m}$ approximating $f$ \cite{prokhorov2015}.

\subsection{A convergence theorem}\label{s:cv}

\begin{theorem}\label{t:prokhorov}
	Fix $B > 0$, $N \ge m \ge 0$, and any sequence $(R_{n,m})_{n \ge N}$ of Prokhorov approximants of $f$. The following are equivalent:
	\begin{alphabetize}
		\item The germ of $f$ at $0$ extends meromorphically to $\abs{z} < B$, with at most $m$ poles counted with multiplicity.
		\item The least singular values satisfy
		\[
		\limsup_{n \to \infty} \sqrt[n]{s_{n+1, m+1}} \le \inv B.
		\]
		\item For every $0 < r < B$,
		\[
		m(r, R_{n,m} - R_{k,m}) \to_{n,k \to \infty} 0.
		\]
		\item For every $0 < r < B$,
		\[
		\sup_{n \ge N} m(r, R_{n,m}) < \infty.
		\]
	\end{alphabetize}
	Under these conditions, if $m \ge 1$, then, for every $n \ge N$ and $0 < r < B$,
	\begin{equation}\label{e:Rerror}
		\begin{aligned}
			m(r, f - R_{n,m})
			&\le 2m \log \pa{1 + 2 \sqrt{2(r + r^{-1} + 2)} \sum_{k \ge n} \pa{r^{k+1} s_{k+1, m+1}}^{\inv{2m}}}.
		\end{aligned}
	\end{equation}
\end{theorem}

The equivalence of (a) and (b) follows from Prokhorov's formula \cite[Thm.~4.1]{prokhorov2015}. Prokhorov also proved $\sigma$-almost uniform convergence \cite[Thm.~4.2]{prokhorov2015}. The present theorem gives a distinct characterization in terms of the proximity function, together with the explicit error bound \cref{e:Rerror}.

For a nonzero polynomial $V$, let
\begin{equation}\label{e:Mr}
	M_r(V) := \exp \pa{\inv{2\pi} \int_0^{2\pi} \log \abs{V(re^{i\theta})} \d\theta}
\end{equation}
be its \emph{Mahler measure} on the circle $\abs{z} = r$
\cite[Prop.~1.5]{brunault2020}. We first record the logarithmic estimate used below.
\begin{lemma}\label{l:logpoly}
	Let $r > 0$ and let $d \ge 2$ and $K \ge 1$ be integers. For each $1 \le k \le K$, let $a_k \ge 0$, let $V_k$ be a nonzero polynomial of degree at most $d$, and let $g_k$ be meromorphic in a neighborhood of the circle $\abs{z} = r$. Assume that, for every $1 \le k \le K$ and every $z$ with $\abs{z} = r$ and $V_k(z) \ne 0$,
	\begin{equation}\label{e:logpolyhyp}
		\abs{g_k(z)} \le a_k \frac{M_r(V_k)}{\abs{V_k(z)}}.
	\end{equation}
	Then
	\begin{equation}\label{e:logpoly}
		m \pa{r, \sum_{k = 1}^K g_k} \le d \log \pa{1 + \sum_{k = 1}^K a_k^{1/d}}.
	\end{equation}
\end{lemma}
\begin{proof}
	We first prove that, for $b_k \ge 0$ and $\abs{\alpha_k} \le 1$,
	\begin{equation}\label{e:logsquares}
		\inv{2\pi} \int_0^{2\pi} \log\pa{1 + \sum_{k = 1}^K \frac{b_k^2}{\abs{e^{i\theta} - \alpha_k}^2}} \d\theta \le 2\log\pa{1 + \sum_{k = 1}^K b_k}.
	\end{equation}
	Without loss of generality, $b_k > 0$. For fixed values of the other parameters, the left-hand side is subharmonic in $\alpha_k$ on the unit disk. Indeed, write its integrand as
	\[
	\log\pa{b_k^2 + \abs{e^{i\theta} - \alpha_k}^2\pa{1 + \sum_{\ell \ne k} \frac{b_\ell^2}{\abs{e^{i\theta} - \alpha_\ell}^2}}} - 2\log\abs{e^{i\theta} - \alpha_k}.
	\]
	The first term is subharmonic in $\alpha_k$, and the second has mean zero by Jensen's formula. The logarithmic singularities are uniformly integrable, so the integral extends continuously to the closed disk. The maximum principle in each parameter therefore reduces \cref{e:logsquares} to $\abs{\alpha_k} = 1$.

	If $\alpha_j = \alpha_k$ for $j \ne k$, replace the corresponding two terms in the sum by a single term with weight $\sqrt{b_j^2 + b_k^2}$. The left-hand side is unchanged, and $\sqrt{b_j^2 + b_k^2} \le b_j + b_k$, so proving the bound after this replacement suffices. Repeating this step, we may assume that the $\alpha_k$ are pairwise distinct.

	Define
	\[
	B := \sum_{k = 1}^K b_k, \qquad H(z) := 1 + \sum_{k = 1}^K \frac{b_k}{1 - \bar{\alpha_k}z}, \qquad u(z) := \Re\frac{1}{H(z)}.
	\]
	Since $\Re H(z) > 0$ for $\abs{z} < 1$, the functions $\log |H|$ and $u$ are harmonic in the unit disk. Moreover, for $\abs{z} = 1$ and $z \ne \alpha_k$,
	\[
	\Re H(z) = 1 + \frac{B}{2}, \qquad \inv{\abs{H(z)}^2} = \frac{2u(z)}{B + 2}.
	\]
	
	Since the $\alpha_k$ are pairwise distinct,
	\[
	\inv{H(z)} = \frac{1 - \bar{\alpha_k}z}{b_k} + O\pa{\abs{z - \alpha_k}^2} \qquad \text{as } z \to \alpha_k.
	\]
	Differentiating the Poisson integral formula for $u(\rho\alpha_k)$ with respect to $\rho$ and letting $\rho \to 1^-$ gives
	\[
	\inv{2\pi} \int_0^{2\pi} \frac{u(e^{i\theta})}{\abs{e^{i\theta} - \alpha_k}^2} \d\theta = -\inv{2}\lim_{\rho \to 1^-}\frac{d}{d\rho}u(\rho\alpha_k) = \inv{2b_k}.
	\]
	The mean-value property gives
	\[
	\inv{2\pi} \int_0^{2\pi} u(e^{i\theta}) \d\theta = u(0) = \inv{1 + B}.
	\]
	Consequently, by the preceding equations,
	\begin{align*}
		&\inv{2\pi} \int_0^{2\pi} \frac{1 + \sum_{k=1}^K b_k^2/\abs{e^{i\theta} - \alpha_k}^2}{\abs{H(e^{i\theta})}^2} \d\theta\\
		={}& \frac{2}{B + 2}\pa{\inv{2\pi} \int_0^{2\pi} u(e^{i\theta}) \d\theta + \sum_{k=1}^K b_k^2 \inv{2\pi} \int_0^{2\pi} \frac{u(e^{i\theta})}{\abs{e^{i\theta} - \alpha_k}^2} \d\theta}\\
		={}& \frac{2}{B + 2}\pa{\inv{1 + B} + \sum_{k=1}^K \frac{b_k}{2}} = \frac{2}{B + 2}\pa{\inv{1 + B} + \frac{B}{2}} \le 1,
	\end{align*}
	and by Jensen's inequality,
	\[
	\inv{2\pi} \int_0^{2\pi} \log\pa{\frac{1 + \sum_{k = 1}^K b_k^2/\abs{e^{i\theta} - \alpha_k}^2}{\abs{H(e^{i\theta})}^2}} \d\theta \le 0.
	\]
	Since $H$ has no zeros in the open unit disk and its poles lie on the circle, Jensen's formula gives
	\[
	\inv{2\pi} \int_0^{2\pi} \log\abs{H(e^{i\theta})} \d\theta = \log\abs{H(0)} = \log(1 + B).
	\]
	Combining these two relations proves \cref{e:logsquares}.

	For each $k$, write $V_k(z) = v_k\prod_{j = 1}^{p_k}(z - \zeta_{k,j})$, where $p_k \le d$. Jensen's formula \cite[Prop.~1.5]{brunault2020} gives
	\[
	\frac{\abs{V_k(re^{i\theta})}}{M_r(V_k)} = \prod_{j = 1}^d \abs{e^{i\theta} - \alpha_{k,j}},
	\]
	where $\alpha_{k,j} = \zeta_{k,j}/r$ if $\abs{\zeta_{k,j}} \le r$, $\alpha_{k,j} = r/\bar\zeta_{k,j}$ otherwise, and $\alpha_{k,j} = 0$ for $p_k < j \le d$. In particular, $\abs{\alpha_{k,j}} \le 1$.
	By \cref{e:logpolyhyp}, Hölder's inequality, and monotonicity of the finite dimensional $\ell^p$ norms,
	\begin{equation}\label{e:logholder}
		\abs{\sum_{k = 1}^K g_k(re^{i\theta})} \le \prod_{j = 1}^d \pa{\sum_{k = 1}^K \frac{a_k}{\abs{e^{i\theta} - \alpha_{k,j}}^d}}^{\inv{d}} \le \prod_{j = 1}^d \sqrt{\sum_{k = 1}^K \frac{a_k^{2/d}}{\abs{e^{i\theta} - \alpha_{k,j}}^2}}.
	\end{equation}
	Take $\log^+$ in \cref{e:logholder}, bound $\log^+ x$ by $\log(1 + x)$ for each sum under a square root, and integrate. Applying \cref{e:logsquares} with weights $a_k^{1/d}$ to each factor proves \cref{e:logpoly}.
\end{proof}

\begin{lemma}[{\cite[Lem.~1.6]{brunault2020}}]\label{l:mahler}
	Let $V(z) = \sum_{i = 0}^m v_i z^i$. Then
	\[
	\abs{v_i} r^i \le \binom{m}{i} M_r(V).
	\]
\end{lemma}

\begin{proof}[of \Cref{t:prokhorov}]
	We prove the equivalence through the chain of implications
	\[
	\textnormal{(a)}\Rightarrow\textnormal{(b)}
	\Rightarrow\textnormal{(c)}\Rightarrow\textnormal{(d)}
	\Rightarrow\textnormal{(a)}.
	\]

	For (a) $\Rightarrow$ (b), choose a polynomial
	\begin{equation}\label{e:S}
		S(z) = \sum_{j = 0}^m a_j z^j, \qquad \norm{S}_2 = 1,
	\end{equation}
	which cancels the poles of $f$ in $\abs{z} < B$. Write
	\begin{equation}\label{e:fS}
		f(z) S(z) = \sum_{k = 0}^{\infty} b_k z^k.
	\end{equation}
	Since $f S$ is holomorphic in $\abs{z} < B$, Cauchy's estimate and \cref{e:maxmod,e:fS} give, for fixed $0 < r < B$,
	\begin{equation}\label{e:bkcauchy}
		\abs{b_k} \le M(r, f S) r^{-k}.
	\end{equation}
	The definitions of $A_{n+1, m+1}$ and the coefficients in \cref{e:fS} give
	\begin{equation}\label{e:AS}
		A_{n+1, m+1} S = (b_{n+1}, \ldots, b_{n+m+1}).
	\end{equation}
	By \cref{e:S}, the coefficient vector of $S$ has Euclidean norm one. Taking $Q=S$ in \cref{e:smin} and applying \cref{e:bkcauchy,e:AS} gives
	\begin{equation}\label{e:scauchy}
		s_{n+1, m+1} \le \sqrt{\sum_{i = 0}^m \abs{b_{n+1+i}}^2} \le M(r, f S) r^{-n - 1} \sqrt{\sum_{i = 0}^m r^{-2i}}.
	\end{equation}
	Taking $n$th roots in \cref{e:scauchy} gives $\limsup_{n \to \infty} \sqrt[n]{s_{n+1, m+1}} \le r^{-1}$. Letting $r \to B$ proves (b).

	For (b) $\Rightarrow$ (c), if $m = 0$, then $R_{n,0}$ is the Taylor polynomial of $f$ of degree $n$, and $s_{n+1,1}$ is the absolute value of the Taylor coefficient of degree $n + 1$. Thus (b) implies that the Taylor series converges uniformly on compact subsets of $\abs{z} < B$, which proves (c). Suppose now that $m \ge 1$. Write
	\begin{equation}\label{e:Ekm}
		f Q_{k,m} - P_{k,m} = z^{k+1} E_{k,m} + O(z^{k+m+2}), \qquad \deg E_{k,m} \le m.
	\end{equation}
	Write
	\begin{equation}\label{e:Ekmexp}
		E_{k,m}(z) = \sum_{i = 0}^m e_{i,k,m} z^i.
	\end{equation}
	Comparing the coefficients of degrees $k + 1$ through $k + m + 1$ in \cref{e:Ekm,e:Ekmexp} shows that $(e_{0,k,m}, \ldots, e_{m,k,m}) = A_{k+1,m+1} Q_{k,m}$. Since $Q_{k,m}$ attains the minimum in \cref{e:smin},
	\begin{equation}\label{e:Ekmnorm}
		\norm{E_{k,m}}_2 = \norm{A_{k+1,m+1} Q_{k,m}}_2 = s_{k+1,m+1}.
	\end{equation}
	The polynomial $P_{k+1,m} Q_{k,m} - P_{k,m} Q_{k+1,m}$ has degree at most $k + m + 1$. Comparing its Taylor coefficients through degree $k + m + 1$ gives
	\begin{equation}\label{e:directdifference}
		\begin{aligned}
			P_{k+1,m} Q_{k,m} - P_{k,m} Q_{k+1,m}
			&= z^{k+1} [Q_{k+1,m} E_{k,m} - z Q_{k,m} E_{k+1,m}]_{\le m}.
		\end{aligned}
	\end{equation}
	Here $[U]_{\le m}$ denotes the sum of the terms of $U$ of degree at most $m$.
	Let $Q(z) = \sum_{i = 0}^m q_i z^i$ have coefficient vector of Euclidean norm one. By \Cref{l:mahler},
	\begin{equation}\label{e:MrQinv}
		\inv{M_r(Q)} = \inv{M_r(Q)} \sqrt{\sum_{i = 0}^m \abs{q_i}^2} \le \inv{M_r(Q)} \sum_{i = 0}^m \abs{q_i} \le \sum_{i = 0}^m \binom{m}{i} r^{-i} = (1 + r^{-1})^m.
	\end{equation}
	The same coefficient bound also gives
	\begin{equation}\label{e:Qsum}
		\sum_{i = 0}^m \abs{q_i} r^i \le M_r(Q) \sum_{i = 0}^m \binom{m}{i} = 2^m M_r(Q).
	\end{equation}
	By the Cauchy--Schwarz inequality and \cref{e:Ekmnorm},
	\begin{equation}\label{e:Ekmsum}
		\sum_{i = 0}^m \abs{e_{i,k,m}} r^i \le \sqrt{\sum_{i = 0}^m \abs{e_{i,k,m}}^2} \sqrt{\sum_{i = 0}^m r^{2i}} \le s_{k+1,m+1} \sum_{i = 0}^m r^i \le (1 + r)^m s_{k+1,m+1}.
	\end{equation}
	Write $Q_{k,m}(z) = \sum_{i = 0}^m q_{i,k,m} z^i$. On $\abs{z} = r$, deleting terms of degree greater than $m$ cannot increase the sum of the absolute values of the terms. The convolution formula, \cref{e:Qsum}, and \cref{e:Ekmsum} therefore give
	\begin{equation}\label{e:QEbound}
		\begin{aligned}
			\abs{[Q_{k+1,m} E_{k,m}]_{\le m}} &\le 2^m (1 + r)^m M_r(Q_{k+1,m}) s_{k+1,m+1}, \\
			\abs{[z Q_{k,m} E_{k+1,m}]_{\le m}} &\le r 2^m (1 + r)^m M_r(Q_{k,m}) s_{k+2,m+1}.
		\end{aligned}
	\end{equation}
	The definition in \cref{e:Mr} implies that $M_r$ is multiplicative. Thus \cref{e:MrQinv} gives
	\begin{equation}\label{e:MrQQbound}
		\max \pa{M_r(Q_{k,m}), M_r(Q_{k+1,m})} \le (1 + r^{-1})^m M_r(Q_{k,m} Q_{k+1,m}).
	\end{equation}
	Since $2^m (1 + r)^m (1 + r^{-1})^m = \pa{2(r + r^{-1} + 2)}^m$, put
	\begin{equation}\label{e:akm}
		a_{k,m} := \pa{2(r + r^{-1} + 2)}^m r^{k+1} \pa{s_{k+1,m+1} + r s_{k+2,m+1}}.
	\end{equation}
	It follows from \cref{e:directdifference,e:QEbound,e:MrQQbound} that
	\begin{equation}\label{e:PQdiff}
		\abs{P_{k+1,m} Q_{k,m} - P_{k,m} Q_{k+1,m}} \le a_{k,m} M_r(Q_{k,m} Q_{k+1,m}).
	\end{equation}
	Dividing \cref{e:PQdiff} by $\abs{Q_{k,m} Q_{k+1,m}}$ gives
	\begin{equation}\label{e:Rdiff}
		\abs{R_{k+1,m} - R_{k,m}} \le a_{k,m} \frac{M_r(Q_{k,m} Q_{k+1,m})}{\abs{Q_{k,m} Q_{k+1,m}}}.
	\end{equation}
	Put
	\begin{equation}\label{e:Enm}
		\mathcal E_{n,m}(r) := \sum_{k \ge n} \pa{r^{k+1} s_{k+1,m+1}}^{\inv{2m}}.
	\end{equation}
	By \cref{e:Enm,e:akm} and the subadditivity of the $\inv{2m}$ power,
	\begin{equation}\label{e:akmsum}
		\sum_{k = n}^{\nu - 1} a_{k,m}^{\inv{2m}} \le 2 \sqrt{2(r + r^{-1} + 2)} \mathcal E_{n,m}(r).
	\end{equation}
	For $\nu > n$, set $g_k := R_{k+1,m} - R_{k,m}$ and $V_k := Q_{k,m} Q_{k+1,m}$. By \cref{e:Rdiff}, the functions $g_k$ satisfy the hypothesis of \Cref{l:logpoly} with $d = 2m$. Therefore, by \cref{e:akmsum},
	\begin{align}
		m(r, R_{\nu,m} - R_{n,m}) &= m\pa{r, \sum_{k = n}^{\nu - 1} g_k} \le 2m \log \pa{1 + \sum_{k = n}^{\nu - 1} a_{k,m}^{\inv{2m}}} \nonumber \\
		&\le 2m \log \pa{1 + 2 \sqrt{2(r + r^{-1} + 2)} \mathcal E_{n,m}(r)}. \label{e:finitetail}
	\end{align}
	For each fixed $r < B$, condition (b) and the root test give
	\[
	\limsup_{k \to \infty} \sqrt[k]{\pa{r^{k+1} s_{k+1,m+1}}^{\inv{2m}}} \le \pa{\frac{r}{B}}^{\inv{2m}} < 1.
	\]
	Thus the series in \cref{e:Enm} converges and $\mathcal E_{n,m}(r) \to 0$ as $n \to \infty$. The right-hand side of \cref{e:finitetail} therefore tends to zero, which proves (c).

	For (c) $\Rightarrow$ (d), fix $0 < r < B$. By (c), choose $k \ge N$ such that $m(r, R_{n,m} - R_{k,m}) \le 1$ for all $n \ge k$. Then
	\[
	m(r, R_{n,m}) \le m(r, R_{n,m} - R_{k,m}) + m(r, R_{k,m}) + \log 2 \le 1 + m(r, R_{k,m}) + \log 2,
	\]
	where $m(r, R_{k,m})$ is finite and independent of $n$, which gives (d).

	For (d) $\Rightarrow$ (a), fix $0 < r < s < B$. Since $\norm{Q_{n,m}}_2 = 1$, the Cauchy--Schwarz inequality gives, for $\abs{z} \le s$,
	\begin{equation}\label{e:Qbound}
		\abs{Q_{n,m}(z)} \le \sum_{\ell = 0}^m \abs{q_{\ell,n,m}} s^\ell \le \sqrt{\sum_{\ell = 0}^m \abs{q_{\ell,n,m}}^2} \sqrt{\sum_{\ell = 0}^m s^{2\ell}} = \sqrt{\sum_{\ell = 0}^m s^{2\ell}}.
	\end{equation}
	Using $P_{n,m} = R_{n,m} Q_{n,m}$, \cref{e:Qbound} implies
	\[
	m(s, P_{n,m}) \le m(s, R_{n,m}) + \log \sqrt{\sum_{\ell = 0}^m s^{2\ell}}.
	\]
	The right-hand side is bounded uniformly in $n$ by condition (d). A standard estimate gives, for $\abs{z} \le r$,
	\begin{equation}\label{e:Pharnack}
		\log^+ \abs{P_{n,m}(z)} \le \frac{s + r}{s - r} m(s, P_{n,m}).
	\end{equation}
	Thus the polynomials $P_{n,m}$ are uniformly bounded on every compact subset of $\abs{z} < B$.

	The coefficient vectors of $Q_{n,m}$ lie on the unit sphere of $\C^{m+1}$. After passing to a subsequence $(n_j)$, they converge to a vector $(q_0, \ldots, q_m)$ of norm one. Put $Q(z) := \sum_{\ell = 0}^m q_\ell z^\ell$. By \cref{e:Pharnack}, Montel's theorem gives a further subsequence and a holomorphic function $F$ on $\abs{z} < B$ such that
	\begin{equation}\label{e:PQlim}
		Q_{n_j,m} \to Q, \qquad P_{n_j,m} \to F
	\end{equation}
	locally uniformly on $\abs{z} < B$. In particular, $Q \ne 0$ and $\deg Q \le m$.

	Fix $k \ge 0$. For all sufficiently large $j$, \cref{e:pnm} gives
	\begin{equation}\label{e:Pcoeff}
		\frac{P_{n_j,m}^{(k)}(0)}{k!} = \sum_{\ell = 0}^{\min(k,m)} q_{\ell,n_j,m} c_{k - \ell}.
	\end{equation}
	Local uniform convergence in \cref{e:PQlim} implies convergence of the derivatives at zero. Hence, on letting $j \to \infty$ in \cref{e:Pcoeff},
	\[
	\frac{F^{(k)}(0)}{k!} = \sum_{\ell = 0}^{\min(k,m)} q_\ell c_{k - \ell}.
	\]
	The right-hand side is the Taylor coefficient of degree $k$ of $Qf$. Therefore $F = Qf$ in a neighborhood of zero. The quotient $F/Q$ is meromorphic on $\abs{z} < B$, agrees with $f$ near zero, and has at most $m$ poles counted with multiplicity. This proves (a).

	It remains to prove \cref{e:Rerror}, for which $m \ge 1$. Fix $n \ge N$ and $0 < r < B$. By \cref{e:PQlim}, $R_{n_j,m} \to F/Q = f$ on the circle $\abs{z} = r$ away from the finitely many zeros of $Q$. Thus the nonnegative integrands $\log^+\abs{R_{n_j,m}(re^{i\theta}) - R_{n,m}(re^{i\theta})}$ converge almost everywhere to $\log^+\abs{f(re^{i\theta}) - R_{n,m}(re^{i\theta})}$. Fatou's lemma and \cref{e:finitetail}, applied with $\nu = n_j$, give
	\[
	m(r, f - R_{n,m}) \le \liminf_{j \to \infty} m(r, R_{n_j,m} - R_{n,m}) \le 2m\log\pa{1 + 2\sqrt{2(r + r^{-1} + 2)}\mathcal E_{n,m}(r)}.
	\]
	The bound is independent of $j$. Substituting \cref{e:Enm} proves \cref{e:Rerror}.
\end{proof}

\subsection{Bounds for the Nevanlinna characteristic}\label{s:bound}

\begin{theorem}\label{t:Tbound}
	Let $B > 0$ and let $n \ge m \ge 0$ be integers. Assume that
	\begin{equation}\label{e:sBbound}
		s_{k,m+1} \le B^{-k}, \qquad k \ge n.
	\end{equation}
	Then the germ of $f$ at $0$ extends meromorphically to $\abs{z} < B$, with at most $m$ poles counted with multiplicity. Choose $0 < \rho < \min(B,R_0)$ and put
	\begin{equation}\label{e:Tconst}
		C_{f,\rho} := \log\pa{2(2 + M(\rho, f))}, \qquad C_\rho := \abs{\log \rho} + \log 128 + 2.
	\end{equation}
	For every $\rho \le r < B$,
	\begin{equation}\label{e:Tbound}
		T(r, f) \le C_{f,\rho} + (n + m + 1) \log\frac{r}{\rho} + \inv{2} \log(n + 1) + C_\rho(m + 1)\log\frac{e(m + 1)}{1 - r/B}.
	\end{equation}
\end{theorem}

\begin{proof}
	The continuation follows from \Cref{t:prokhorov}, since \cref{e:sBbound} implies its condition (b). There are no poles in $\abs{z} \le \rho$, so, for $\rho \le r < B$,
	\begin{equation}\label{e:Nbound}
		N(r, f) \le m \log\frac{r}{\rho}.
	\end{equation}

	We first rescale the disk of radius $\rho$ to the unit disk. Put $g(z) := f(\rho z)$ and let $D_\rho := \operatorname{diag}(1,\rho,\ldots,\rho^m)$. The Taylor coefficient of degree $j$ of $g$ is $\rho^j c_j$, and hence
	\[
	A_{k+1,m+1}(g) = \rho^{k+1} D_\rho A_{k+1,m+1}(f) D_\rho^{-1}.
	\]
	Let $v$ be a unit vector attaining the least singular value of $A_{k+1,m+1}(f)$. Taking $Q = D_\rho v/\norm{D_\rho v}_2$ in \cref{e:smin} gives
	\begin{equation}\label{e:srhobound}
		s_{k+1,m+1}(g) \le \rho^{k+1} \frac{\norm{D_\rho A_{k+1,m+1}(f)v}_2}{\norm{D_\rho v}_2} \le \rho^{k+1} \norm{D_\rho}_2 \norm{D_\rho^{-1}}_2 s_{k+1,m+1}(f).
	\end{equation}
	Here the matrix norm is the operator norm induced by the Euclidean norm. Since $\norm{D_\rho}_2 \norm{D_\rho^{-1}}_2 = \exp(m\abs{\log \rho})$ and $\rho/B < 1$, \cref{e:sBbound,e:srhobound} give
	\begin{equation}\label{e:srhoBbound}
		s_{k+1,m+1}(g) \le e^{m\abs{\log \rho}} \pa{\frac{\rho}{B}}^k, \qquad k \ge n.
	\end{equation}

	In the rest of the proof, let $R_{k,m} := P_{k,m}/Q_{k,m}$ be Prokhorov approximants of $g$, with $\norm{Q_{k,m}}_2 = 1$. Since $P_{n,m}$ is the Taylor polynomial of degree $n$ of $gQ_{n,m}$, Parseval's identity gives
	\begin{equation}\label{e:Pnorm}
		\norm{P_{n,m}}_2^2 \le \inv{2\pi} \int_0^{2\pi} \abs{g(e^{i\theta})Q_{n,m}(e^{i\theta})}^2 \d\theta \le M(\rho, f)^2 \norm{Q_{n,m}}_2^2 = M(\rho, f)^2.
	\end{equation}
	Fix $\rho \le r < B$. For $\abs{z} = r/\rho$, the Cauchy--Schwarz inequality and \cref{e:Pnorm} give
	\begin{equation}\label{e:Prhobound}
		\abs{P_{n,m}(z)} \le \norm{P_{n,m}}_2 \sqrt{\sum_{j = 0}^n \pa{\frac{r}{\rho}}^{2j}} \le M(\rho, f) \sqrt{n + 1} \pa{\frac{r}{\rho}}^n.
	\end{equation}
	By \cref{e:MrQinv}, $\inv{M_{r/\rho}(Q_{n,m})} \le (1 + \rho/r)^m \le 2^m$. Thus \cref{e:Prhobound} yields, away from the zeros of $Q_{n,m}$,
	\begin{equation}\label{e:Rrhobound}
		\abs{R_{n,m}} \le 2^m M(\rho, f) \sqrt{n + 1} \pa{\frac{r}{\rho}}^n \frac{M_{r/\rho}(Q_{n,m})}{\abs{Q_{n,m}}}.
	\end{equation}

	The coefficient estimates giving \cref{e:akm,e:Rdiff} remain valid for $m = 0$. Applying them to $g$ at radius $r/\rho$, and using $2(r/\rho + \rho/r + 2) \le 8r/\rho$ and \cref{e:srhoBbound}, gives
	\begin{equation}\label{e:akmrhobound}
		\begin{aligned}
			a_{k,m} &\le 8^m e^{m\abs{\log \rho}} \pa{\frac{r}{\rho}}^{m+1} \pa{\frac{r}{B}}^k \pa{1 + \frac{r}{B}} \\
			&\le 2 \cdot 8^m e^{m\abs{\log \rho}} \pa{\frac{r}{\rho}}^{m+1} \pa{\frac{r}{B}}^k.
		\end{aligned}
	\end{equation}
	In particular,
	\begin{equation}\label{e:Rrhodiff}
		\abs{R_{k+1,m} - R_{k,m}} \le a_{k,m} \frac{M_{r/\rho}(Q_{k,m}Q_{k+1,m})}{\abs{Q_{k,m}Q_{k+1,m}}}, \qquad k \ge n.
	\end{equation}

	To bound the scalar factor in \cref{e:Rrhobound} and the numbers $a_{k,m}$ in \cref{e:akmrhobound} together, put
	\begin{equation}\label{e:V}
		V := 2 \cdot 8^m \sqrt{n + 1} \pa{\frac{r}{\rho}}^{n+1} \pa{2 + M(\rho, f)} e^{m\abs{\log \rho}}.
	\end{equation}
	Since $r/\rho \ge 1$ and $n \ge m$, we have $V \ge 1$, $2^m M(\rho, f)\sqrt{n + 1}(r/\rho)^n \le V$, and $a_{k,m} \le V(r/B)^k$. For $\nu > n$, summing a geometric series gives
	\begin{equation}\label{e:akmrhosum}
		\pa{2^m M(\rho, f)\sqrt{n + 1}\pa{\frac{r}{\rho}}^n}^{\inv{2(m+1)}} + \sum_{k = n}^{\nu - 1} a_{k,m}^{\inv{2(m+1)}} \le \frac{2V^{\inv{2(m+1)}}}{1 - (r/B)^{\inv{2(m+1)}}},
	\end{equation}
	where we bounded the sum over $k \ge n$ by the sum over $k \ge 0$.

	The identity $R_{\nu,m} = R_{n,m} + \sum_{k = n}^{\nu - 1}(R_{k+1,m} - R_{k,m})$ expresses $R_{\nu,m}$ as a finite sum to which \Cref{l:logpoly} applies with $d = 2(m + 1)$. Indeed, the denominators in \cref{e:Rrhobound,e:Rrhodiff} have degree at most $2m$. Consequently, \cref{e:akmrhosum} gives
	\begin{equation}\label{e:mRbound}
		m\pa{\frac{r}{\rho}, R_{\nu,m}} \le 2(m + 1)\log\pa{1 + \frac{2V^{\inv{2(m+1)}}}{1 - (r/B)^{\inv{2(m+1)}}}} \le \log V + 2(m + 1)\log\frac{8(m + 1)}{1 - r/B}.
	\end{equation}
	For the second inequality, use $V \ge 1$, the bound $1 + 2x \le 4x$ for $x \ge 1$, and
	\[
	1 - \frac{r}{B} = \pa{1 - \pa{\frac{r}{B}}^{\inv{2(m+1)}}}\sum_{j = 0}^{2m+1}\pa{\frac{r}{B}}^{j/(2(m+1))} \le 2(m + 1)\pa{1 - \pa{\frac{r}{B}}^{\inv{2(m+1)}}}.
	\]

	By \cref{e:srhoBbound}, \Cref{t:prokhorov} applies to $g$ on $\abs{z} < B/\rho$. The compactness argument giving \cref{e:PQlim} provides a subsequence along which $R_{\nu,m}$ tends to $g$ on $\abs{z} = r/\rho$ except at the finitely many zeros of a nonzero limiting denominator. Fatou's lemma applied to \cref{e:mRbound} therefore gives
	\begin{equation}\label{e:mfbound}
		m(r, f) = m\pa{\frac{r}{\rho}, g} \le \log V + 2(m + 1)\log\frac{8(m + 1)}{1 - r/B}.
	\end{equation}

	Expanding \cref{e:V} and adding \cref{e:Nbound} to \cref{e:mfbound} gives
	\begin{equation}\label{e:Tboundexp}
		\begin{aligned}
			T(r, f) &\le C_{f,\rho} + (n + m + 1)\log\frac{r}{\rho} + \inv{2}\log(n + 1) \\
			&\quad + m\pa{\log 8 + \abs{\log \rho}} + 2(m + 1)\log\frac{8(m + 1)}{1 - r/B}.
		\end{aligned}
	\end{equation}
	The last two terms of \cref{e:Tboundexp} are at most $(m + 1)(\abs{\log \rho} + \log 512 + 2\log((m + 1)/(1 - r/B)))$. Since $\log((m + 1)/(1 - r/B)) \ge 0$ and $\log 4 \le 2$, this is at most $C_\rho(m + 1)\log(e(m + 1)/(1 - r/B))$, proving \cref{e:Tbound}.
\end{proof}

\section{Order from the coefficients}\label{s:order}

This section applies the preceding results to formulas for the order of meromorphic functions in terms of their Taylor coefficients.

\subsection{Generalized order}

The classical notion of order for entire functions, formally introduced by Borel in 1897 \cite{borel1897a}, was motivated by \emph{Hadamard's factorization theorem}. But depending on the context, a different notion of order may be more suitable to measure different scales of growth, something that had already been suggested by Borel \cite[p.~396]{borel1897a}. To motivate our definition of generalized order, first consider the following well-known facts about meromorphic functions \cite[Chap.~I, Thm.~6.4{;} p.~21]{goldberg2008}. 
\begin{proposition}\label{p:transc} \hfill
	\begin{enumerate}
		\item A meromorphic function $f$ is transcendental if and only if
		\[
		\frac{T(r, f)}{\log r} \to_{r\to\infty} \infty.
		\]
		\item The characteristic $T(r, f)$ is a convex function of $\log r$.
	\end{enumerate}
\end{proposition}

More precisely, Clunie \cite{clunie1965} proved that every increasing function $c(r)$ that is convex in $\log r$ and satisfies $c(r)/\log r \to \infty$ is asymptotically equivalent to $T(r, f)$ for some entire function $f$. This motivates the definition of a generalized order, which measures the residual growth after division by $\log r$:
\begin{align}\label{e:rhoL}
	\rho_L(f)
	:={}& \inf\{\varrho>0 : T(r, f) = O(e^{\varrho L(\log r)}\log r)\} \\
	={}& \limsup_{r\to\infty}\frac{\log^+ (T(r, f)/\log r)}{L(\log r)} \nonumber \\
	\in{}& [0, \infty], \nonumber
\end{align}
where $L$ belongs to the growth class $\fL$ we define later. This notion of order is unchanged when $L$ is replaced by an asymptotically equivalent scale. When $\log \log r = o(L(\log r))$, this definition reduces to a common definition of generalized order of the form
\begin{equation}\label{e:rhoLold}
	\rho_L(f) = \limsup_{r\to\infty}\frac{\log^+ T(r, f)}{L(\log r)},
\end{equation}
which includes the classical notion of order \eqref{e:rho} and Lindelöf's power-iterated $\log$ scale \cite{fernandezarias2010,lindelof1902a}.

Another special case is worth noting: when $f$ is nonconstant and $L(t) = \log t$, \cref{e:rhoLold} defines a quantity $\lambda$ known as the \emph{logarithmic order} \cite{chern2005,fernandezarias2010}. Since $T(r, f) = o(\log r)$ implies that $f$ is constant, the logarithmic order of a nonconstant meromorphic function cannot be smaller than $1$. With our definition \eqref{e:rhoL},
\[
\rho_{\log}(f) = \lambda - 1,
\]
which attains all values of $[0, \infty]$. For constant functions, both quantities defined by these formulas are zero. The quantity $\lambda - 1$ is ubiquitous in \cite{chern2005}, which further motivates \cref{e:rhoL}.

There is also a practical reason for retaining the factor $\log r$: we have not been able to remove it from our estimates.

\medskip

Let $\fL$ be the family of scales $L : [t_0,\infty)\to(0,\infty)$ that are unbounded and nondecreasing and satisfy
\begin{equation}\label{e:grclass}
	L\pa{t + te^{-cL(t)}} \sim L(t), \qquad \text{for every } c > 0,
\end{equation}
as $t \to \infty$. Our class generalizes Sheremeta's growth class \cite{sheremeta1970}, which requires
\begin{equation}\label{e:sheremeta}
	L\pa{t + t\epsilon(t)} \sim L(t), \qquad \text{for every } \epsilon(t) \to 0.
\end{equation}
For instance, $\fL$ contains the exponential, which does not satisfy \cref{e:sheremeta}. The peculiar definition of $\fL$ is motivated by the maximality theorem in \Cref{s:max}. Throughout, we adopt the convention $\log(1/0) = L(\infty) = \infty$.

The following lemma characterizes $\fL$ by perturbations of the argument of $L$ and by the existence of a differentiable representative.
\begin{lemma}\label{l:grclass}
	Let $L : [t_0,\infty)\to(0,\infty)$ be unbounded and nondecreasing. The following are equivalent:
	\begin{alphabetize}
		\item $L \in \fL$.
		\item For every $c > 0$ and every real function $\delta(t) = O(te^{-cL(t)})$,
		\begin{equation}\label{e:Lstable}
			L(t + \delta(t)) \sim L(t).
		\end{equation}
		\item There exists a positive differentiable nondecreasing function $\cL \sim L$ such that
		\begin{equation}\label{e:Lderiv}
			\log^+ \pa{\frac{t \cL'(t)}{\cL(t)}} = o(\cL(t)).
		\end{equation}
	\end{alphabetize}
	Moreover, $\cL$ may be chosen in $C^\infty$ with $\cL'(t) > 0$ and $\cL(t) > L(t)$ for all sufficiently large $t$.
\end{lemma}

In particular, \cref{e:Lderiv} offers a sufficient condition for membership of $\fL$ that is independent of $c$.

\begin{proof}
	\hfill
	\begin{itemize}
		\item[(b) $\Rightarrow$ (a)] Taking $\delta(t) := te^{-cL(t)}$ in \cref{e:Lstable} gives \cref{e:grclass}.

		\item[(a) $\Rightarrow$ (c)] We may assume $t_0 > 0$. By (a), choose strictly increasing thresholds $T_j \to \infty$, $j \ge 1$, such that
		\begin{equation}\label{e:LTj}
			L\pa{t\pa{1 + e^{-L(t)/j}}} \le \pa{1 + \inv j}L(t), \qquad t \ge T_j.
		\end{equation}
		Starting at $t_1 \ge T_1 > 0$, define
		\begin{equation}\label{e:tk}
			t_{k+1} := t_k\pa{1 + e^{-L(t_k)/j}}, \qquad T_j \le t_k < T_{j+1}.
		\end{equation}
		Then $t_k \to \infty$. Indeed, if $t_k \le T$ for all $k$, monotonicity and $j \ge 1$ would give
		\[
		t_{k+1} - t_k = t_k e^{-L(t_k)/j} \ge t_1 e^{-L(T)} > 0,
		\]
		a contradiction. Consequently, $L(t_k) \to \infty$, and the index $j$ determined by $T_j \le t_k < T_{j+1}$ tends to infinity with $k$. By \cref{e:LTj,e:tk},
		\begin{equation}\label{e:Lratio}
			1 \le \frac{L(t_{k+1})}{L(t_k)} \le 1 + \inv{j} \to 1.
		\end{equation}
		Take a smooth nondecreasing function $\theta : [0,1] \to {[0,1]}$ equal to $0$ on a neighborhood of $0$ and to $1$ on a neighborhood of $1$. On each interval $t_k \le t \le t_{k+1}$, define
		\begin{equation}\label{e:cL}
			\cL(t) := L(t_{k+1}) + (L(t_{k+2}) - L(t_{k+1}))\theta\pa{\frac{t - t_k}{t_{k+1} - t_k}} + \frac{t}{1+t}.
		\end{equation}
		For $t < t_1$, set $\cL(t) := L(t_2) + \frac{t}{1+t}$. All positive-order derivatives of $\theta$ vanish at the endpoints, so these formulas define a positive smooth function with $\cL'(t) > 0$. By monotonicity and \cref{e:tk,e:cL}, for $t_k \le t \le t_{k+1}$,
		\begin{equation}\label{e:cLbound}
			L(t_k) \le L(t) \le L(t_{k+1}) < \cL(t) < L(t_{k+2}) + 1.
		\end{equation}
		Here $0 < \frac{t}{1+t} < 1$. Thus \cref{e:Lratio,e:cLbound} give $\cL(t) \sim L(t_k) \sim L(t)$ uniformly for $t_k \le t \le t_{k+1}$, and $\cL(t) > L(t)$.

		Using $\frac{t}{(1+t)^2} \le 1$ and writing $C := \sup |\theta'|$, we have, uniformly for $t_k \le t \le t_{k+1}$ with $T_j \le t_k < T_{j+1}$,
		\begin{equation}\label{e:cLderiv}
			\frac{t\cL'(t)}{\cL(t)} \le \inv{L(t_k)} + C\frac{L(t_{k+2})}{L(t_k)}\pa{1 + e^{L(t_k)/j}} = O\pa{e^{L(t_k)/j}}.
		\end{equation}
		Consequently, \cref{e:cLderiv} gives
		\[
		\log^+\pa{\frac{t\cL'(t)}{\cL(t)}} \le \frac{L(t_k)}{j} + O(1) = o(\cL(t)).
		\]
		This proves \cref{e:Lderiv} and the "moreover" assertion.

		\item[(c) $\Rightarrow$ (b)] Fix $c > 0$ and a real function $\delta(t) = O(te^{-cL(t)})$, and put $H(t) := e^{-\frac{c}{4}\cL(t)}$. By \cref{e:Lderiv}, eventually,
		\[
		t\abs{H'(t)} = \frac{c\cL(t)}{4}\frac{t\cL'(t)}{\cL(t)}e^{-\frac{c}{4}\cL(t)} \le \frac{c\cL(t)}{4}e^{-\frac{c}{8}\cL(t)} \to 0.
		\]
		In particular, $\abs{H'(t)} \le \inv{t}$ eventually. Put $s := t + \delta(t)$. Since $s \sim t$, the mean value theorem and $\cL(t) \le 2L(t)$ for all large $t$ give
		\begin{equation}\label{e:Hratio}
			\abs{\frac{H(s)}{H(t)} - 1} \le \frac{2\abs{\delta(t)}}{tH(t)} = O\pa{e^{-cL(t)+\frac{c}{4}\cL(t)}} \to 0.
		\end{equation}
		Taking logarithms in \cref{e:Hratio} yields
		\begin{equation}\label{e:cLdiff}
			\cL(s) - \cL(t) = -\frac{4}{c}\log\frac{H(s)}{H(t)} \to 0.
		\end{equation}
		Since $\cL(t) \to \infty$, \cref{e:cLdiff}, and $L \sim \cL$ imply
		\[
		L(s) \sim \cL(s) \sim \cL(t) \sim L(t).
		\]
		This is \cref{e:Lstable} and proves (b).
	\end{itemize}
\end{proof}

We record a specialization that will be useful in the subsequent sections.

\begin{corollary}\label{c:radius}
	Let $L \in \fL$ and $q > 0$, and put $D := \lceil e^{qL(\log r)}\rceil$ and $B := r\pa{1 + \inv{D}}$. Then, as $r \to \infty$,
	\[
	L(\log B) \sim L(\log r), \qquad \frac{\log D}{L(\log B)} \to q.
	\]
\end{corollary}
\begin{proof}
	Since $0 \le \log B - \log r \le \inv{D} \le e^{-qL(\log r)}$, \Cref{l:grclass} gives $L(\log B) \sim L(\log r)$. The second conclusion follows from $\log D \sim qL(\log r)$.
\end{proof}

\subsection{Entire functions}

For each fixed matrix size, the coefficient formulas take the following form.
\begin{theorem}\label{t:rhoEL}
	Assume $L\in\fL$. For every entire function $f$ and every fixed matrix size $m\ge1$,
	\begin{align}
		\rho_L(f) &= \limsup_{n\to\infty} \frac{\log n}{L(\inv{n} \log\inv{s_{n,m}})} \label{e:rhoELs}\\
		&= \limsup_{n\to\infty} \frac{\log n}{L(\inv{nm} \log\inv{\abs{H_{n,m}}})}. \label{e:rhoELH}
	\end{align}
\end{theorem}

\begin{proof}
	The case $f \equiv 0$ is immediate. Fix $m \ge 1$. By \cref{e:HMbound} at each fixed radius, the arguments of $L$ tend to infinity. By \Cref{l:grclass}, we may therefore assume that $L$ is smooth and strictly increasing. Let $\varrho_s$ and $\varrho_H$ denote the two limit superiors in the statement. By \cref{e:sHbound}, $\varrho_s \le \varrho_H$ so it suffices to prove $\varrho_H \le \rho_L(f) \le \varrho_s$.

	Suppose first that $\rho_L(f) < q$ and choose $\epsilon > 0$ such that $\rho_L(f) + 2\epsilon < q$. Choose $R$ by $qL(\log R) = \log n$ and put $r := R/\pa{1 + e^{-\epsilon L(\log R)}}$. By \cref{e:rhoL}, $T(R, f) \le C\log R e^{(q - 2\epsilon)L(\log R)}$ for some $C > 0$. Comparison with the Poisson integral of $\log^+\abs{f}$ on $\abs{z} = R$ and Harnack's inequality give
	\begin{equation}\label{e:TMbound}
		T(r, f) \le \log^+ M(r, f) \le \frac{R + r}{R - r} T(R, f).
	\end{equation}
	Since $\log(R/r) \le e^{-\epsilon L(\log R)}$ and $(R + r)/(R - r) = 1 + 2e^{\epsilon L(\log R)}$, \cref{e:HMbound,e:TMbound} give, for some constant $C_m > 0$ and all sufficiently large $n$,
	\begin{equation}\label{e:HdecayE}
		\inv{nm}\log\inv{\abs{H_{n,m}}} \ge \log r - \frac{\log^+ M(r, f)}{n} - \frac{\log(m!)}{nm} \ge \log R\pa{1 - C_m e^{-\epsilon L(\log R)}}.
	\end{equation}
	By \Cref{l:grclass}, \cref{e:HdecayE} implies that $L(\inv{nm}\log\inv{\abs{H_{n,m}}}) \ge (1 - o(1))\log n/q$. Thus $\varrho_H \le q$, and letting $q \to \rho_L(f)$ proves the first inequality when the order is finite. If $\rho_L(f) = \infty$, the inequality $\varrho_H \le \rho_L(f)$ holds as well.

	For the converse, suppose that $\varrho_s < q$. At radius $r$, put $N := \lceil e^{qL(\log r)}\rceil$ and $B := r\pa{1 + \inv{N}}$. By \Cref{c:radius} with $D = N$, $\log N/L(\log B) \to q > \varrho_s$. Thus, for all sufficiently large $r$ and every $n \ge N$, the definition of $\varrho_s$ gives
	\[
	L\pa{\inv{n}\log\inv{s_{n,m}}} > L(\log B).
	\]
	Monotonicity of $L$ implies $s_{n,m} \le B^{-n}$. Fix $\rho > 0$. Since $\inv{1 - r/B} = N + 1$ and $m$ is fixed, \Cref{t:Tbound}, applied with $n$ replaced by $N$ and $m$ by $m - 1$, gives
	\begin{equation}\label{e:TgrowthE}
		T(r, f) = O(N\log r).
	\end{equation}
	Since $N = O(e^{qL(\log r)})$, \cref{e:rhoL,e:TgrowthE} give $\rho_L(f) \le q$. Let $q \to \varrho_s$. If $\varrho_s = \infty$, the inequality $\rho_L(f) \le \varrho_s$ holds as well. Together with $\varrho_s \le \varrho_H$, this proves both formulas.
\end{proof}

The generalized order of an entire function can therefore also be computed from its maximum modulus.
\begin{corollary}\label{c:rhoLM}
	Let $L \in \fL$. For every entire function $f$,
	\begin{equation}\label{e:rhoLM}
		\rho_L(f) = \limsup_{r \to \infty} \frac{\log^+\pa{\log^+ M(r, f)/\log r}}{L(\log r)}.
	\end{equation}
\end{corollary}
\begin{proof}
	The inequality $T(r, f) \le \log^+ M(r, f)$ gives the lower bound in \cref{e:rhoLM}. Suppose that $\rho_L(f) < \infty$ and fix $q > \rho_L(f)$. The case $m = 1$ of the proof of \Cref{t:rhoEL} gives $\abs{c_n} \le B^{-n}$ for $n \ge N$, with $N := \lceil e^{qL(\log r)}\rceil$ and $B := r\pa{1 + \inv{N}}$. Fix $\rho > 0$. Cauchy's estimate on the circle of radius $\rho$ and a geometric series give, for $r \ge \rho$,
	\begin{equation}\label{e:Mtaylor}
		M(r, f) \le M(\rho, f)N\pa{\frac{r}{\rho}}^N + N + 1.
	\end{equation}
	Thus \cref{e:Mtaylor} gives $\log^+ M(r, f) = O(N\log r) = O(e^{qL(\log r)}\log r)$. Letting $q \to \rho_L(f)$ gives the upper bound. The infinite case follows from the lower bound.
\end{proof}

In particular, with $L$ as the identity function, \cref{e:rhoELH} reduces to Edrei's formula \eqref{e:edrei}. When $m=1$ and $L$ satisfies \cref{e:sheremeta}, \cref{e:rhoELs} follows from a theorem of Sheremeta \cite{sheremeta1970} and \Cref{c:rhoLM}. Both cases further reduce to the Poincaré--Hadamard formula \eqref{e:rhoE}.

\subsection{Generalized type}

Assuming $\rho_L(f) < \infty$, one can define the type as
\begin{equation}
	\sigma_L(f) := \limsup_{r\to\infty} \frac{T(r, f)}{e^{\rho_L(f)L(\log r)} \log r},
\end{equation}
which depends on the finer equivalence class induced by the equivalence relation $L_1(t) = L_2(t) + o(1)$.

Equivalently to \Cref{p:transc}, a meromorphic function $f$ is rational if and only if $T(r, f) = O(\log r)$. In that case, a classical result \cite[p.~33]{goldberg2008} gives
\[
\frac{T(r, f)}{\log r} \to_{r\to\infty} \max(\deg P, \deg Q),
\]
where $P$ and $Q$ are respectively the numerator and denominator of $f$. Therefore, when $f$ is rational, $\rho_L(f) = 0$ and $\sigma_L(f) = \max(\deg P, \deg Q)$. Without the logarithmic factor, the type would not give any meaningful information, which further motivates its presence.

\medskip

In \Cref{s:max} we prove a maximality theorem for the class $\fL$. We present a related maximality result for generalized type, established by Filevych \cite{filevych2001}, further developed in connection with coefficient formulas \cite{hlova2014}, and generalized by Hlova and Filevych \cite{hlova2015}. We recall it here in a coefficient form, which makes the analogy with our results more apparent.

Let $g : [t_0,\infty)\to(0,\infty)$ be nondecreasing, with $g(t)/t \to \infty$, and define
\[
g^\sharp(x) := \inv{x g^\circ(1/x)} = \inf_{t > x} \frac{g(t)}{t - x},
\]
where $g^\circ$ is the \emph{polar} of $g$ \cite[Sec.~15]{rockafellar1970}.
For a nonzero entire function $f$, define a different notion of generalized type by
\[
\tau_g(f) := \limsup_{r\to\infty} \frac{\log M(r, f)}{g(\log r)},
\]

\begin{theorem}[Hlova--Filevych, 2015]\label{t:HF}
	Assume that $g$ is continuous, nondecreasing and satisfies $g(t)/t \to \infty$. The equation
	\[
	\tau_g(f) = \limsup_{n\to\infty} \frac{n}{g^\sharp\pa{\inv{n}\log\inv{\abs{c_n}}}}
	\]
	holds for every nonzero entire function $f$ if and only if
	\[
	\log g^\sharp(t) = o(g(t)).
	\]	
\end{theorem}

Our definition of $g^\sharp$ agrees with their construction as the inverse of $g^*(t)/t$, where $g^*$ is the convex conjugate of $g$ \cite[Sec.~12]{rockafellar1970}. The theorem can then be deduced from \cite[Thm.~1, eq.~(4)]{hlova2015}.

\subsection{Maximality of the family of scales}\label{s:max}

The special case
\begin{equation}\label{e:rhoEL}
	\rho_L(f) = \limsup_{n\to\infty} \frac{\log n}{L(\inv{n} \log\inv{\abs{c_n}})}
\end{equation}
of \Cref{t:rhoEL} was shown by Sheremeta \cite{sheremeta1970} under the condition of \cref{e:sheremeta}. We show that the condition \eqref{e:grclass} is not only less restrictive but defines the maximal class $\fL$ for which \cref{e:rhoEL} holds.
\begin{theorem}
	Let $L : [t_0,\infty)\to(0,\infty)$ be unbounded and nondecreasing. \cref{e:rhoEL} holds for every entire function $f$ if and only if $L\in\fL$.
\end{theorem}

\begin{proof}
	Sufficiency is \cref{e:rhoEL}. For necessity, assume that $L$ is unbounded and nondecreasing. Then for every $c > 0$, $L(t + te^{-c L(t)}) \ge L(t)$. Now suppose that \cref{e:grclass} fails. Then for some $\varrho > 0$, $b > 1$, and $s_j \to \infty$, $L$ satisfies $L(t_j) \ge b L(s_j)$ where $t_j := s_j(1 + e^{-\varrho L(s_j)})$ and $j\ge 1$. Passing to a subsequence, we may assume without loss of generality that $s_j \ge j$ and the integers $n_j := \floor{e^{\varrho L(s_j)}}$ are strictly increasing. Define
	\[
	f(z) := \sum_{n = 0}^\infty c_n z^n, \quad c_n := \begin{cases}
		e^{-n_j s_j} &\text{if } n = n_j \\
		0 &\text{otherwise.}
	\end{cases}
	\]
	This series is entire because $\sqrt[n]{\abs{c_n}} \to 0$. Since $\inv{n_j}\log\inv{|c_{n_j}|} = s_j$ and $\log n_j \sim \varrho L(s_j)$, we have
	\begin{equation}\label{e:rhomaxcoeff}
		\limsup_{n\to\infty} \frac{\log n}{L(\inv{n}\log\inv{\abs{c_n}})} = \varrho.
	\end{equation}

	For sufficiently large $t$, put $A(t) := e^{\varrho L(t)/b}$. We claim that, for every $j$,
	\begin{equation}\label{e:njbound}
		(n_j-1)(t-s_j)\le tA(t)
	\end{equation}
	\begin{itemize}
		\item If $t \le s_j$, this is immediate.
		\item If $s_j < t < t_j$, then
		\[
		(n_j-1)(t-s_j) < n_j(t_j-s_j) = \floor{e^{\varrho L(s_j)}} s_j e^{-\varrho L(s_j)} \le s_j < t < tA(t).
		\]
		\item If $t \ge t_j$,
		\[
		(n_j-1)(t-s_j) < n_j t \le e^{\varrho L(s_j)} t
		\le t e^{\varrho L(t_j)/b} \le tA(t).
		\]
	\end{itemize}

	The coefficients are nonnegative. Thus, using \cref{e:njbound} and $s_j \ge j$, we obtain
	\[
	M(e^t, f) = f(e^t) = \sum_{j=1}^\infty e^{n_j(t-s_j)}\le e^{tA(t)}\sum_{j=1}^\infty e^{t-s_j} < e^{tA(t)+t}.
	\]
	Consequently,
	\[
	\log^+ M(r, f) \le \log(r) (e^{\varrho L(\log r)/b} + 1),
	\]
	and thus $\rho_L(f) \le \varrho/b < \varrho$ which, together with \cref{e:rhomaxcoeff}, contradicts \cref{e:rhoEL}.
\end{proof}

\subsection{Meromorphic functions}

\begin{theorem}\label{t:rhoL}
	Assume $L\in\fL$. For every globally meromorphic $f$ holomorphic at $0$,
	\begin{align}
		\rho_L(f) &= \limsup_{\substack{m\to\infty \\ n/m\to\infty}} \frac{\log m}{L(\inv{n} \log\inv{s_{n,m}})} \\
		&= \limsup_{\substack{m\to\infty \\ n/m\to\infty}} \frac{\log m}{L(\inv{nm} \log\inv{\abs{H_{n,m}}})} \\
		&= \limsup_{n,m\to\infty} \frac{\log\min(n,m)}{L(\inv{nm} \log\inv{\abs{H_{n,m}}})}.
	\end{align}
\end{theorem}

Here,
\begin{equation}\label{e:limsupdef}
	\limsup_{\substack{m\to\infty \\ n/m\to\infty}} a_{n,m} := \lim_{M,K\to\infty} \sup_{\substack{m\ge M \\ n\ge Km}} a_{n,m}.
\end{equation}

\begin{proof}
	The case $f \equiv 0$ is immediate. Applying \cref{e:pomme} at each fixed radius shows that the arguments of $L$ tend to infinity as $n,m \to \infty$. Thus all the quotients are unchanged under within the equivalence class of asymptotic equivalence. By \Cref{l:grclass}, we may assume that $L$ is smooth and strictly increasing. Let $\varrho_s$ and $\varrho_H$ denote the first and last limit superiors, respectively. By \cref{e:sHbound}, the middle $\limsup$ lies between them, so it suffices to prove $\varrho_H \le \rho_L(f) \le \varrho_s$.

	For the first inequality, suppose that $\rho_L(f) < \infty$ and fix $\rho_L(f) < p < q$. Choose $R$ by $qL(\log R) = \log\min(n,m)$. By \cref{e:rhoL}, $T(R, f) \le C\log R e^{pL(\log R)}$ for some $C > 0$. Since $(n + m)/(nm) \le 2/\min(n,m)$, \cref{e:pomme} gives
	\begin{equation}\label{e:Hdecay}
		\inv{nm}\log\inv{\abs{H_{n,m}}} \ge \log R - \frac{2T(R, f)}{\min(n,m)} \ge \log R\pa{1 - 2Ce^{-(q - p)L(\log R)}}.
	\end{equation}
	Applying \Cref{l:grclass} to \cref{e:Hdecay}, as in the argument following \cref{e:HdecayE}, gives $\varrho_H \le q$. Letting $q \to \rho_L(f)$ proves the first inequality. If $\rho_L(f) = \infty$, it is automatic.

	For the second inequality, suppose that $\varrho_s < \infty$ and fix $\varrho_s < q_0 < q$. By \cref{e:limsupdef}, there exist $M$ and $K \ge 1$ such that
	\begin{equation}\label{e:sLbound}
		L\pa{\inv{n}\log\inv{s_{n,m}}} \ge \frac{\log m}{q_0}, \qquad m \ge M, \quad n \ge Km.
	\end{equation}
	For each sufficiently large $r$, put $m := \lceil e^{qL(\log r)}\rceil$, $N := \lceil Km\rceil$, and $B := r\pa{1 + \inv{m}}$. By \Cref{c:radius} with $D = m$, $\log m/L(\log B) \to q > q_0$. Thus $\log m/q_0 > L(\log B)$ eventually, and \cref{e:sLbound} implies $s_{n,m} \le B^{-n}$ for every $n \ge N$. Fix $0 < \rho < R_0$ and apply \Cref{t:Tbound} with $n$ replaced by $N$ and $m$ by $m - 1$. Since $N = O(m)$ and $\inv{1 - r/B} = m + 1$, \cref{e:Tbound} gives
	\[
	T(r, f) = O\pa{m\pa{\log r + \log m}}.
	\]
	Dividing by $\log r$, taking $\log^+$, and using $\log m \sim qL(\log r)$ and $\log(1 + L(\log r)) = o(L(\log r))$ yields $\rho_L(f) \le q$ by \cref{e:rhoL}. Let $q \to \varrho_s$. If $\varrho_s = \infty$, the second inequality is automatic. This proves all three equalities.
\end{proof}

Formula \cref{e:rhoH} also holds for Laurent series, with the determinants formed from coefficients that vanish for $n < \nu \in \Z$. If $f$ has a pole of order $-\nu > 0$ at the origin, put $g(z) := z^{-\nu} f(z)$. Then $g$ is holomorphic there and
\[
H_{n,m}(f) = H_{n-\nu,m}(g), \qquad \rho(f) = \rho(g),
\]
after suitably extending the notion of order. The result follows by applying \cref{e:rhoH} to $g$, since a fixed shift in $n$ does not change the $\limsup$.

\printbibliography

\end{document}